\documentclass[12pt]{amsart}
\usepackage{SSdefn}

\title{Stillman's question for twisted commutative algebras}
\author{Karthik Ganapathy}
\address{Department of Mathematics, University of Michigan, Ann Arbor, MI}
\email{\href{mailto:karthg@umich.edu}{karthg@umich.edu}}
\thanks{The author was partially supported by a Donald J.~Lewis Fellowship at the University of Michigan.}
\urladdr{\url{http://www-personal.umich.edu/~karthg/}}
    

\date{}
\begin{document}
\begin{abstract}
Let $\bA_{n, m}$ be the polynomial ring $\Sym(\bC^n \otimes \bC^m)$ with the natural action of $\GL_m(\bC)$.
We construct a family of $\GL_m(\bC)$-stable ideals $J_{n, m}$ in $\bA_{n, m}$, each equivariantly generated by one homogeneous polynomial of degree $2$. Using the Ananyan--Hochster principle, we show that the regularity of this family is unbounded. This negatively answers a question raised by Erman--Sam--Snowden on a generalization of Stillman's conjecture.
\end{abstract}
\maketitle
\section{Introduction}
{Let $K$ be a field. Stillman's conjecture asserts that the projective dimension of an ideal in $K[x_1, x_2, \ldots, x_n]$ generated by $r$ homogeneous polynomials of degree $\leq d$ can be bounded purely in terms of $r$ and $d$ and (crucially) independent of $n$, the number of variables. Caviglia showed that Stillman's conjecture is equivalent to the same statement with ``projective dimension" replaced by ``regularity" (see \cite{msbound}, Theorem 2.4).   
A positive answer to Stillman's conjecture first appeared in \cite{ahsmall} and two more proofs appeared in \cite{essbig}.}  The purpose of this short note is to answer the following question raised by Erman--Sam--Snowden generalizing Stillman's conjecture:
\begin{question}[see Question 5.6 in \cite{essgen}]\label{q:stillmantca}
Let $\bA_n$ be the twisted commutative algbera $\Sym(\bC^n \otimes \bC^{\infty})$. Is there a function $N(d, r)$ such that the regularity of $I$ is less than $N(d, r)$ for every $\GL$-ideal $I$ equivariantly generated by $r$ homogeneous polynomials of degree $\leq d$ in any $\bA_n$?
\end{question}

A twisted commutative algebra (tca) $\bA$ is a commutative $\bC$-algebra equipped with an action of $\GL_{\infty} = \bigcup_{n\geq 1} \GL_n(\bC)$ under which $\bA$ forms a \emph{polynomial} $\GL_{\infty}$ representation (see \S 8 of \cite{sstca} for details). 
A $\GL$-ideal of $\bA$ is an ideal of the $\bC$-algebra $\bA$ stable under the action of $\GL_{\infty}$. We say that a $\GL$-ideal $I$ is \emph{equivariantly generated} by elements $f_1, f_2, \ldots, f_r$ in $\bA$ if the smallest $\GL$-ideal containing $f_1, f_2, \ldots, f_r$ is $I$. The only tcas considered in this note are the free tca's generated in degree one, namely, $\bA_n = \Sym(\bC^n \otimes \bC^{\infty})$. 

Resolutions of nonzero proper $\GL$-ideals in $\bA_n$ are always infinite. Therefore, the ``projective dimension" version of Stillman's conjecture does not apply. However, these resolutions exhibit strong finiteness properties. For instance, the regularity of an ideal is finite, and each linear strand of the Betti table is ``finitely generated" in a suitable sense (see \S 7 in \cite{ssgl2}).  

One proof of Stillman's conjecture in \cite{essbig} crucially relies on the fact that the parameter space of ideals generated by $r$ generators of degree $\leq d$ in any $K[x_1, x_2, \ldots, x_n]$ is a $\GL$-noetherian topological space by \cite{dtop}. The space parametrizing $\GL$-ideals equivariantly generated by $r$ homogeneous polynomials of degree $\leq d$ in any $\bA_n$ is also a $\GL$-noetherian topological space by \cite{dtop} (see also \S 2 in \cite{essgen}).

Given these parallels between ideals in polynomial rings and $\GL$-ideals, it is reasonable to expect a positive answer to Question~\ref{q:stillmantca}. On the contrary, we show that the regularity cannot be bounded even when $d = 2$ and $r=1$, that is, for $\GL$-ideals equivariantly generated by one degree $2$ polynomial. We now explain the key idea behind our counterexample.

A tca (and its $\GL$-ideals) can be viewed as a polynomial functor on $\bC$-vector spaces (see \S8 of \cite{sstca}). If $I_n \subset \bA_n$ is a $\GL$-ideal, then treating them as polynomial functors, we set $\bA_{n,m} = \bA_n(\bC^m)$ and $I_{n, m} = I_n(\bC^m)$. $\bA_{n,m}$ is the polynomial ring in $n m$ variables with an action of $\GL_m(\bC)$ and $I_{n, m}$ is a $\GL_m(\bC)$-stable ideal of $\bA_{n, m}$. 
Furthermore, if one takes a minimal free resolution of $I_n$ as an $\bA_n$-module and evaluates it at $\bC^m$, one obtains a minimal free resolution of $I_{n, m}$ as an $\bA_{n, m}$-module. This implies that the regularity of $I_{n, m}$ is less than or equal to the regularity of $I_n$. 
In the next section, we explicitly construct a family of $\GL$-ideals $J_n \subset \bA_n$ each equivariantly generated by a homogeneous polynomial of degree $2$ and prove the following result about the family $J_n$:
\begin{theorem}\label{thm:mainresult}
There exists a function $g\colon \bN \rightarrow \bN$ such that $\limsup_{n \rightarrow \infty} g(n) = \infty$ and $J_{n,g(n)}$ is generated (as a usual ideal of $\bA_{n, g(n)}$) by $\frac{g(n)(g(n)+1)}{2}$ homogeneous polynomials of degree $2$ which form a regular sequence.
\end{theorem}
From this, it is easy to see that the regularity of the family $J_n$ is unbounded. Indeed, since $J_{n,g(n)}$ is generated by a regular sequence of quadrics, its regularity is seen to be exactly $\frac{g(n)(g(n)+1)}{2}$ using the Koszul complex. This implies that the regularity of $J_n$ is at least $\frac{g(n)(g(n)+1)}{2}$, which is an unbounded function of $n$. Therefore, assuming Theorem~\ref{thm:mainresult}, we have proved:
\begin{corollary}
The answer to Question~\ref{q:stillmantca} is no.
\end{corollary}
We now explain what we believe is the subtlety of {Theorem~\ref{thm:mainresult}. Fix an $n$ and let $I_n \subset \bA_{n}$ be a $\GL$-ideal generated in degree greater than $1$. The property that $I_{n, m}$ is generated by a regular sequence can only hold for small values of $m$: the (Krull) dimension of the ring $\bA_{n, m}$ is $nm$, which is linear in $m$, and the number of generators of the ideal $I_{n, m}$ grows to the order of (at least) $m^2$, but the length of a regular sequence can be at most the dimension of the ring.
However, Theorem~\ref{thm:mainresult} tells us that the small values of $m$ for which $J_{n, m}$ is generated by a regular sequence can be made as large as we want, provided we consider all $n$.}

After a preliminary version of this paper appeared, Steven Sam pointed out that the ideals $J_n$ were previously considered in Section~4 of \cite{sswlit} and in fact, one can take $g(n) = \lfloor n/2 \rfloor$ by Lemma~4.2 in loc. cit. (which is stronger than our main result). However, we include our original proof of Theorem~\ref{thm:mainresult} as our approach is considerably different, and is related to the methods used in \cite{essbig}. Our method also generalizes to all degrees $d > 1$ (see Remark~\ref{rmk:gendeg}).
\begin{remark}\label{rmk:ext}
Peeva asked an analogue of Stillman's question for exterior algebras. This was answered negatively in \cite{mcext}, where the counterexample is also a family of principal ideals generated in degree 2. {The regularity is shown to be large by proving that the generation degrees of the first syzygy is unbounded; in the family we construct, we prove that regularity is large by ``isolating" a non-zero diagonal of large length in the Betti table. It might be interesting to see whether the first syzygies of our counterexample are generated in unbounded degree as well.}
\end{remark}
\subsection*{Acknowledgements} {The author thanks Andrew Snowden for introducing him to the problem, and patiently explaining the foundations of representation stability. The author also thanks Anna Brosowsky for answering many questions about Macaulay2 and Singular.}

\section{Proof of Theorem~\ref{thm:mainresult}}
Let $\bA_n$ be the tca  $\Sym(\bC^n \otimes \bC^{\infty})$ which we identify as $\bC[x_{i, j} | 1 \leq i \leq n, j \in \bN]$ with the natural $\GL_{\infty}$ action. Under this identifcation, $\bA_{n, m} = \bA_n(\bC^m)$ is the polynomial ring $\bC[x_{i, j}|1 \leq i \leq n, 1\leq j \leq m]$. There is a grading on $\bA_n$ given by $\deg(x_i)= 1$ which descends to $\bA_{n, m}$ as well.

\begin{definition}
The \textbf{strength} of a homogeneous polynomial of $f \in \bA_{n, m}$ with $\deg(f) \geq 2$ is the minimal integer $k$ such that $f$ can be written as $\sum_{i=1}^{k+1} g_i h_i$ for homogeneous polynomials $g_i, h_i \in \bA_{n, m}$ with $\deg(g_i), \deg(h_i) > 0$ for all $i$. The \emph{collective strength} of a collection of homogeneous polynomials $f_1, f_2, \ldots, f_n$ of degree $\geq 2$ is the minimal strength of a non-trivial homogeneous $\bC$-linear combination.
\end{definition}
\begin{remark}\label{rmk:rank}
The strength of a quadratic polynomial $Q$ is $\lceil \frac{\rank(Q)}{2} \rceil - 1$, where $\rank(Q)$ is the rank of $Q$ considered as a quadratic form.
\end{remark}

For each $n$, let $J_n$ be the $\GL$-ideal of $\bA_n$ equivariantly generated by $x_{1, 1}^2 + x_{2, 1}^2 + \ldots + x_{n, 1}^2$. Let $J_{n,m} := J_n(\bC^m) \subset \bA_{n, m}$. 
It is easy to see that the ideal $J_{n,m}$ is minimally generated (as a usual ideal of $\bA_{n,m}$) by $f_{n, i,j} = x_{1,i}x_{1,j} + x_{2,i}x_{2,j}+ \ldots x_{n, i}x_{n,j}$ for $1 \leq i \leq j \leq m$. 
By the preceding remark, the strength of $f_{n, i, j}$ is {exactly $\lceil \frac{n}{2} \rceil - 1$ if $i = j$, or exactly $n-1 $ if $i \ne j$}
Let $\bF_{n,m}$ denote the sequence $f_{n,1, 1}, f_{n,1, 2}, \ldots f_{n,m, m}$ in $\bA_{n,m}$. The length of $\bF_{n, m}$ is {$\binom{m+1}{2}$}. It is clear that if $\bF_{n, m}$ is a regular sequence, then so is $\bF_{n, m-1}$.

\begin{proof}[Proof of Theorem~\ref{thm:mainresult}]
Let $g(n)$ to be the maximum $m$ such that $\bF_{n, m}$ is a regular sequence in $\bA_{n, m}$. Since the length of a regular sequence (consisting of homogeneous polynomials) in a graded ring is at most the dimension of the ring, we have $g(n) \leq 2n-1$, and in particular, $g(n)$ is finite. We will prove that $g(n)$ is unbounded. If not, then there exists a $k$ such that $\bF_{n,k}$ is not a regular sequence in $\bA_{n, k}$ for all $n$. So for each $n$, we have $\binom{k+1}{2}$ homogeneous polynomials of degree $2$ which do not form a regular sequence. By Theorem 1.6 in \cite{essbig}, there exists an $N$ such that the collective strength of $\bF_{n,k}$ is less than $N-1$ for all $n$. We will show that the collective strength of $\bF_{2N, k}$ is at least $N-1$.

Let $Q = \sum_{1 \leq i \leq j \leq k} \alpha_{i, j} f_{2N, i, j}$ be an arbitrary $\bC$-linear combination of the elements of $\bF_{2N, k}$ such that strength of $Q$ is less than $N-1$.
We can write $Q = \sum_{i=1}^{N-1} G_i H_i$, where $G_i$ and $H_i$ are linear forms. If $\alpha_{i, i} \ne 0$ for some $i$, by substituting $x_{s,t} = 0$ for all variables with $t \ne i$ we get, $\alpha_{i, i} f_{2N, i, i} = \widetilde{Q} = \sum_{i=1}^{N-1} \widetilde{G_i} \widetilde{H_i}$ where $\widetilde{Q}$ is the polynomial obtained by substituting $x_{s,t} = 0$ if $t \ne i$ in $Q$ (and similarly for $G_i$ and $H_i$). If all of the $\alpha_{i, i} = 0$, then choose $i, j$ such that $\alpha_{i, j} \ne 0$. As before, by substituting $x_{s,t} = 0$ if $t \ne i, j$, we get, $\alpha_{i, j} f_{2N, i, j} = \widetilde{Q} = \sum_{i=1}^{N-1} \widetilde{G_i} \widetilde{H_i}$, where $\widetilde{Q}$ is the polynomial obtained by substituting $x_{s,t} = 0$ if $t \ne i, j$ (and similarly for $G_i$ and $H_i$). Therefore, we see that the strength of either $f_{2N, i, i}$ or $f_{2N, i, j}$ is less than $N-1$, but the strength of both these polynomials is at least $N-1$, a contradiction.
\end{proof}
\begin{remark}\label{rmk:gendeg}
Our counterexample generalizes to all degrees $d > 1$. Let $J_n$ be the $\GL$-ideal equivariantly generated by $f_n = x_{1,1}^d + x_{2,1}^d + \ldots + x_{n, 1}^d$ in $\bA_n$. {The strength of the collection $\{f_n\}_{n \in \bN}$ is unbounded (this follows from Example 1.5(b) in \cite{essbig}).} Using this, one can suitably modify the above proof (and theorem statement) to obtain an unbounded function $g(n)$ as before. Therefore, the regularity of this collection of ideals, each equivariantly generated by one degree $d$ polynomial, is also unbounded.
\end{remark}
{Our proof suggests that instead of trying to bound the regularity, one could attempt to bound the generation degree of the $i$-th syzygy for a fixed $i$.
More precisely,
\begin{question}
Is there a function $N(i, d, r)$ such that the generation degree of the $i$-th syzygy of $I$ is less than $N(i, d, r)$ for every $\GL$-ideal $I$ equivariantly generated by $r$ homogeneous polynomials of degree $\leq d$ in any $\bA_n$?
\end{question}
As a first step, one could try to determine whether the first syzygies of our counterexample are generated in bounded degree or not (see Remark~\ref{rmk:ext}}).


\begin{thebibliography}{DGPS20}
\bibitem[AH19]{ahsmall}
Tigran Ananyan and Melvin Hochster.
\newblock {Small subalgebras of polynomial rings and Stillman’s Conjecture}.
\newblock {\em Journal of the American Mathematical Society}, 33(1):291–309,
  Oct 2019.

\bibitem[DGPS20]{DGPS}
Wolfram Decker, Gert-Martin Greuel, Gerhard Pfister, and Hans Sch\"onemann.
\newblock {\sc Singular} {4-1-1} --- {A} computer algebra system for polynomial
  computations.
\newblock \url{http://www.singular.uni-kl.de}, 2020.

\bibitem[Dra19]{dtop}
Jan Draisma.
\newblock {Topological Noetherianity of polynomial functors}.
\newblock {\em Journal of the American Mathematical Society}, 32(3):691–707,
  Apr 2019.

\bibitem[ESS19a]{essbig}
Daniel Erman, Steven~V. Sam, and Andrew Snowden.
\newblock {Big polynomial rings and Stillman’s conjecture}.
\newblock {\em Inventiones mathematicae}, 218(2):413–439, May 2019.

\bibitem[ESS19b]{essgen}
Daniel Erman, Steven~V Sam, and Andrew Snowden.
\newblock {Generalizations of Stillman’s Conjecture via Twisted Commutative
  Algebra}.
\newblock {\em International Mathematics Research Notices}, 2019.

\bibitem[GS]{M2}
Daniel~R. Grayson and Michael~E. Stillman.
\newblock Macaulay2, a software system for research in algebraic geometry.
\newblock Available at \url{http://www.math.uiuc.edu/Macaulay2/}.

\bibitem[McC19]{mcext}
Jason McCullough.
\newblock {Stillman’s question for exterior algebras and Herzog’s
  conjecture on Betti numbers of syzygy modules}.
\newblock {\em Journal of Pure and Applied Algebra}, 223(2):634–640, Feb
  2019.

\bibitem[MS13]{msbound}
Jason McCullough and Alexandra Seceleanu.
\newblock Bounding projective dimension.
\newblock In {\em Commutative algebra}, pages 551--576. Springer, New York,
  2013.

\bibitem[SS12]{sstca}
Steven~V Sam and Andrew Snowden.
\newblock {Introduction to twisted commutative algebras}.
\newblock {\em arXiv preprint arXiv:1209.5122}, 2012.

\bibitem[SS19]{ssgl2}
Steven~V Sam and Andrew Snowden.
\newblock {$\GL$-equivariant modules over polynomial rings in infinitely many
  variables. II}.
\newblock {\em Forum of Mathematics, Sigma}, 7, 2019.

\bibitem[SSW13]{sswlit}
Steven~V Sam, Andrew Snowden, and Jerzy Weyman.
\newblock Homology of littlewood complexes.
\newblock {\em Selecta Mathematica}, 19(3):655--698, 2013.

\end{thebibliography}
\end{document}